\documentclass[11pt]{article}
\usepackage{amssymb}
\usepackage{amsthm}
\usepackage{amsmath}
\usepackage{graphicx}

 \newtheorem{thm}{Theorem}[section]

 \newtheorem{prop}[thm]{Proposition}
 \theoremstyle{definition}
 
 \newtheorem{rem}[thm]{Remark}
 \numberwithin{equation}{section}

\theoremstyle{definition}

\theoremstyle{remark}

\begin{document}
\title{Structure of Helicity and Global Solutions of Incompressible
Navier-Stokes Equation}
\author{Zhen Lei\footnote{School of Mathematical Sciences; LMNS and Shanghai Key
  Laboratory for Contemporary Applied Mathematics, Fudan
  University, Shanghai 200433, P. R. China. {\it Email:
  leizhn@yahoo.com}}
\and Fang-Hua Lin\footnote{Courant Institute of Mathematics, New
  York University, USA; and Institute of Mathematical Sciences of NYU-ECNU at
  NYU-Shanghai, PRC. {\it Email: linf@cims.nyu.edu}}\and  Yi Zhou\footnote{School of Mathematical
Sciences; LMNS and Shanghai Key
  Laboratory for Contemporary Applied Mathematics, Fudan
  University, Shanghai 200433, P. R. China. {\it Email:
  yizhou@fudan.edu.cn}}}
\date{\today}
\maketitle

\begin{abstract}
In this paper we derive a new energy identity for the
three-dimensional incompressible Navier-Stokes equations by
a special structure of helicity. The new energy functional
is critical with respect to the natural scalings of the Navier-Stokes
equations. Moreover, it is conditionally coercive. As an application we
construct a family of finite energy smooth solutions to the
Navier-Stokes equations whose critical norms can be arbitrarily large.
\end{abstract}


\section{Introduction}%----------------------------introduction

The question of whether a solution of the three-dimensional (3D)
incompressible Navier-Stokes equations can develop a finite time
singularity from smooth initial data with finite energy is one of
the Millennium Prize problems \cite{Fefferman1}. The only
known coercive \textit{a priori} estimate is the Leray-Hopf energy
estimate which implies that the 3D Navier-Stokes equations are
supercritical with respect to its natural scalings. The latter may be
the essence of difficulties of this long standing open problem.

In this paper, by the virtue of a special structure of Helicity, we
derive a new \textit{a priori} energy estimate which is critical
with respect to the natural scalings for the 3D Navier-Stokes
equations. This new energy functional is coercive for a class of
initial data. Based on this \textit{a priori} estimate, a family of
finite energy global smooth and large solutions can then be
constructed. Current known examples of large smooth solutions to the
3D Navier-Stokes equations often assume both the axial symmetricity and the
vanishing of swirl component of the velocity, see \cite{L}, \cite{UY} and
\cite{HLL}.

Let us recall that the incompressible Navier-Stokes
equations in $\mathbb{R}_+ \times \mathbb{R}^3$ are:
\begin{equation}\label{NS}
\begin{cases}
\partial_tu + u\cdot\nabla u + \nabla p = \nu\Delta u,\quad t > 0, x \in \mathbb{R}^3,\\[-3mm]\\
\nabla\cdot u = 0,\quad t > 0, x \in \mathbb{R}^3,
\end{cases}
\end{equation}
where $u$ is the velocity field of the fluid, $p$ is the scalar
pressure and the constant $\nu$ is the viscosity.  To solve the
Navier-Stokes equations \eqref{NS} in $\mathbb{R}_+ \times
\mathbb{R}^3$, one assumes that the initial data
\begin{equation}\nonumber
u(0, x) = u_0(x)
\end{equation}
are divergence-free and possess certain regularity.

The known \textit{a priori} Leray-Hopf energy estimate satisfied by
classical solutions of \eqref{NS} is as follows:
\begin{equation}\label{EnergyL}
\sup_{t > 0}\|u(t, \cdot)\|_{L^2} \leq \|u_0\|_{L^2},\quad
\nu\int_0^\infty\|\nabla u(t, \cdot)\|_{L^2}^2dt \leq
\frac{1}{2}\|u_0\|_{L^2}^2.
\end{equation}
Recall the natural scalings of the Navier-Stokes equations: if $(u,
p)$ solves \eqref{NS}, so does $(u^\lambda, p^\lambda)$ for any
$\lambda > 0$, where
\begin{equation}\label{scaling}
u^\lambda(t, x) = \lambda u(\lambda t, \lambda x),\quad p^\lambda(t,
x) = \lambda^2p(\lambda t, \lambda x).
\end{equation}
As usual, we assign each $x_i$ a positive dimension $1$, $t$ a positive
dimension $2$, $u$ a negative dimension $-1$ and $p$ a negative dimension
$-2$. A simple dimensional analysis shows that all energy norms in
\eqref{EnergyL} have positive dimensions, and thus the
Navier-Stokes equations are \textit{supercritical} with respect to
the natural scalings. An example of dimensionless norm is
$L^\infty_t(\dot{H}^{\frac{1}{2}})$, and it will be related to
discussions below.

Denote
\begin{equation}\nonumber
D = \sqrt{-\Delta}.
\end{equation}
Our starting point is the following new energy identity.
\begin{thm}[Structure of Helicity]\label{thm-coercive}
Let $u \in L^\infty(0, T; H^1(\mathbb{R}^3)) \cap L^2(0, T;
H^2(\mathbb{R}^3))$ solve the incompressible Navier-Stokes equation
\eqref{NS}. For each $t \in [0, T)$, decompose $u(t, \cdot)$ as in
\eqref{decom}. Then one has
\begin{equation}\label{criticalenergy}
E_c(u_+) = E_c(u_-) + c_0, \quad \forall\ t \in [0, T),
\end{equation}
where the constant $c_0$ is given by
$$c_0 = \frac{1}{2}\big(\|u_{0+}\|_{\dot{H}^{\frac{1}{2}}}^2 - \|u_{0-}\|_{\dot{H}^{\frac{1}{2}}}^2\big),$$
and the critical energy $E_c(u)$ is defined as
\begin{equation}\nonumber
E_c(u) = \frac{1}{2}\|D^{\frac{1}{2}}u(t, \cdot)\|_{L^2}^2 +
\nu\int_0^t\|D^{\frac{1}{2}}\nabla u(s, \cdot)\|_{L^2}^2ds.
\end{equation}
\end{thm}

The above energy identity which is based on the
special structure of helicity gives us an \textit{a priori}
estimate. What may be crucial is that this \textit{a
priori} estimate \eqref{criticalenergy} is critical with respect to
the natural scalings \eqref{scaling} of the Navier-Stokes equations.
Moreover, for the initial data so that either  $u_+$ or $u_-$
dominates,  this \textit{a priori} estimate \eqref{criticalenergy}
becomes coercive. The proof of Theorem \ref{thm-coercive} will be
presented in Section 2.

As an application of Theorem \ref{thm-coercive}, we shall construct a
family of finite energy smooth large solutions for the 3D
incompressible Navier-Stokes equations. Define $n(\xi)$ as a measurable vector field
which is smooth except for finite many singular points and satisfies $\xi\cdot n(\xi)$, $|n(\xi)|=1$. For $0 < \delta < 1$, let
\begin{equation}\label{condi-1}
\alpha \in \mathcal{S}(\mathbb{R}^3),\quad {\rm supp}\ \alpha
\subset \{1 - \delta < |\xi| < 1 + \delta\}.
\end{equation}
Assume further that $\alpha$ vanishes in a neighbourhood of the singular points of $n(\xi)$ and
\begin{equation}\label{condi-2}
A = \int_{1 - \delta}^{1 + \delta}\sup_{\omega \in
\mathcal{S}^2}\big(|\alpha(\lambda\omega)| +
|\nabla\alpha(\lambda\omega)|\big)d\lambda < \infty.
\end{equation}
 Define
\begin{equation}\label{data-1}
g(x) = \int_{1 - \delta < |\xi| < 1 +
\delta}\big(n(\xi)\sin(x\cdot\xi) + |\xi|^{- 1}\xi\times
n(\xi)\cos(x\cdot\xi)\big)\alpha(\xi)d\xi.
\end{equation}
Let $\chi \in C_0^\infty(\mathbb{R}^3)$ be a cut-off function such
that
\begin{equation}\label{1-1}
\chi \equiv 1\ {\rm for}\ |x| \leq 1,\quad \chi \equiv 0\ {\rm for}\
|x| \geq 2,\quad |\nabla^k\chi| \leq 2\ (0 \leq k \leq 2)
\end{equation}
and
\begin{equation}\label{1.1}
\chi_M(x) = \chi(\frac{x}{M}),\quad M > 0.
\end{equation}
We have the following theorem:
\begin{thm}\label{thm-1}
Let $g$ is given in \eqref{data-1} where $\alpha$ is given in
\eqref{condi-1} and satisfies \eqref{condi-2}. Let $\chi$ be any
standard cut-off function satisfying \eqref{1-1}-\eqref{1.1}. There
exist positive constants $M \geq \delta^{-\frac{1}{2}} \gg 1$ such
that the 3D incompressible Navier-Stokes equations \eqref{NS} with
the initial data $u_0 = h_0 + \chi_Mg$ are globally well-posed
provided that $\|h_0\|_{H^1} \leq M^{-\frac{1}{2}}$.
\end{thm}

\begin{rem}
An important property of the initial data is that it satisfies $g_- = 0$ in the above
Theorem, which leads one to believe that $u_+$  would dominate in the
evolution of the Navier-Stokes flows. A typical example for $h_0$ in both Theorem \ref{thm-1} and \ref{thm} can be computed by $u_0 = \nabla\times (\nabla\times (\chi_Mg)) = h_0 + \chi_Mg$. We also note that $A$ can be
arbitrarily large (but finite) in the above Theorem. The latter implies
that $g \in L^\infty$ since $\widehat{g} \in L^1$ (see \eqref{3-1}
for details). However, $\widehat{g}$ may not be an $L^p$-function
for any $1<p\leq 3$ and thus our data may not be small in any critical
spaces including $\dot{B}^{-1}_{\infty,\infty}$. Besides, the
integral interval $[1 - \delta, 1 + \delta]$ can be changed to
$\rho[1 - \delta, 1 + \delta]$ for arbitrary positive $\rho$ by
changing $d\lambda$ to $\lambda^2d\lambda$, with various appropriate
modifications. One may even consider the initial data to be a suitable combination
of finitely many $g$'s, say, $g = \sum g_i$, with each $\widehat{g_i}$ supported on
$\{\rho_i(1 - \delta) \leq |\xi| \leq \rho_i(1 + \delta)\}$. Of course we need impose
extra conditions on $\delta$ and $\rho_i$'s. For instance, $\delta \ll \max_i\{\rho_i/\rho_{i + 1}\} \ll 1$ would work.
\end{rem}

\begin{rem}
There are several other constructions of large, finite energy and smooth
solutions for 3D Navier-Stokes equations. Readers may find the following
articles to be informative and relevant: Chemin-Gallagher-Paicu \cite{Chemin}, Hou-Lei-Li
\cite{HLL} and references therein. We shall emphasize that in these
references there may be smallness assumptions imposed on certain
dimensionless norms of unknowns or a part of unknowns. For instance,
in \cite{Chemin}, the $L^2$ norm of $|D_z^{\frac{1}{2}}u_1| +
|D_z^{\frac{1}{2}}u_2| + |D_{xy}^{-1}D_z^{3/2}u_3|$ is small. Here
$\widehat{D_zf}(\xi) = |\xi_3|\widehat{f}(\xi)$ and
$\widehat{D_{xy}^{-1}f}(\xi) = (|\xi_1| +
|\xi_2|)^{-1}\widehat{f}(\xi)$.
\end{rem}

A simple and more typical example of $g$ in the initial data is
\begin{equation}\label{data-2}
g_0(x) = \int_{\mathbb{S}^2}\big(n(\xi)\sin(x\cdot\xi) + |\xi|^{-
1}\xi\times n(\xi)\cos(x\cdot\xi)\big)\beta(\xi)d\sigma_\xi.
\end{equation}
Here $\beta \in C^1(\mathcal{S}^2)$ is a given function which vanishes in a neighbourhood of the singular point of $n(\xi)$. This
turns out to be the steady state Beltrami flow, which has already been observed
in the book of Majda and Bertozzi in \cite{Majda} (see also
\cite{CM}):
$$\nabla\cdot g_0 = 0,\quad \nabla\times g_0 = g_0.$$
We refer to section 3 below for a more detailed discussion. Here we can formulate the following theorem:
\begin{thm}\label{thm}
Let $\beta \in C^1(\mathcal{S}^2)$ and $\chi$ be any standard
cut-off function satisfying \eqref{1-1}-\eqref{1.1}. There exists a
large positive constant $M$ such that the 3D incompressible
Navier-Stokes equations \eqref{NS} with the initial data $u_0 = h_0
+ \chi_M g_0$ are globally well-posed provided that $M \geq M_0$ and
$\|h_0\|_{H^1} \leq M^{-\frac{1}{2}}$.
\end{thm}

The proof of Theorem \ref{thm} is exactly the same as that for
Theorem \ref{thm-1}. Our proof of Theorem \ref{thm-1} is
elementary and it is based on a perturbation argument along with a standard cut-off
technique. A key point is a decay estimate in the spatial directions of
such family of initial data. Let us briefly explain the main idea involved.
We let $v$ be obtained from the heat flow with initial data $g$ or $g_0$. Write the solution as $u = h +
v\chi_M$. Then we try to solve for $h$. Note that $h$ is not divergence-free (see equation
\eqref{h-eqn}). Then main difficulty in solving for $h$ is that the "force term" may not be small.
Indeed, it is easy to check that one of the forcing terms in
$h$-equation (see \eqref{h-eqn} and \eqref{f}) is
$\nabla(\frac{1}{2}\chi_M^2|v|^2)$, which is not small in
$L^2_t(\dot{H}^{-\frac{1}{2}})$ (or $L^2_t(L^3)$ norm of
$\frac{1}{2}\chi_M^2|v|^2$). Thus the standard parabolic estimate
doesn't give an $L^2_t(\dot{H}^{\frac{3}{2}})\cap
L^\infty_t(\dot{H}^{\frac{1}{2}})$ estimate of the perturbation $h$
from $v\chi_M$. In addition, $\|\nabla\cdot
h\|_{L^\infty_t(\dot{H}^{-\frac{1}{2}})}$ (and thus $\|
h\|_{L^\infty_t(\dot{H}^{\frac{1}{2}})}$) may not be small.
Details of these will be discussed in section 4.

The remaining part of the paper is organized as follows. We will
first prove Theorem \ref{thm-coercive} by the virtue of the
structure of the helicity in section 2. In section 3 we study the
decay properties of the initial data given in \eqref{data-1} and
\eqref{data-2}. Then we prove Theorem \ref{thm-1} and Theorem
\ref{thm} in section 4.

\section{Structure of Helicity}

It is well-known that the helicity $\int u\cdot\omega dx$ is
conserved in time for 3D incompressible Euler equations. However,
due to the presence of dissipation, the helicity is not conserved
for 3D incompressible Navier-Stokes equations. It is not clear how to
make use of such a quantity without positivity of its integrand even
in the case of Euler equations.

In this section we will explore a structure of the helicity which is
inherent by smooth solutions to the 3D incompressible Euler or
Navier-Stokes equations. Our proof of Theorem \ref{thm-coercive} is
based on a strongly orthogonal decomposition of the velocity vector
which is stated in Proposition \ref{prop-2}. Similar conclusions
were studied earlier by P. Constantin  and A. Majda for
incompressible Euler equations in \cite{CM}.

Let $u$ be a divergence-free vector field. We make the following
decomposition:
\begin{equation}\label{decom}
u = u_+ + u_-,
\end{equation}
where
\begin{equation}\nonumber
u_+ = \frac{1}{2}\big(u + D^{-1}\nabla\times u\big)
\end{equation}
and
\begin{equation}\nonumber
u_- = \frac{1}{2}\big(u - D^{-1}\nabla\times u\big).
\end{equation}
We have the following proposition:
\begin{prop}\label{prop-1}
Let $u \in H^1(\mathbb{R}^3)$ be a 3D divergence-free vector field
and be decomposed into $u_+$ and $u_-$ as in \eqref{decom}. Then the
following identities hold:
\begin{equation}\nonumber
\nabla\times u_+ = Du_+,\quad \nabla\times u_- = - D u_-.
\end{equation}
\end{prop}
\begin{proof}
We only need to show the first identity. The second one is similar.
The proof is straightforward. In fact, due to the divergence-free
property of $u$, it is clear that
\begin{eqnarray}\nonumber
&&\nabla\times u_+ = \frac{1}{2}\big(\nabla\times u +
  D^{-1}\nabla\times\nabla\times u\big)\\\nonumber
&&= \frac{1}{2}\big(\nabla\times u -
  D^{-1}\Delta  u\big)\\\nonumber
&&= \frac{1}{2}D\big(D^{-1}\nabla\times u +
  u\big)\\\nonumber
&&= Du_+.
\end{eqnarray}
\end{proof}

The following proposition shows that $u_+$ and $u_-$ are strongly
orthogonal to each other.
\begin{prop}\label{prop-2}
Let $m, k \geq 0$ be any integers and $u \in C^m([0, T),
H^k(\mathbb{R}^3))$. Suppose that for each $t \in [0, T)$, $u(t,
\cdot)$ is divergence-free. Decompose $u(t, \cdot)$ into $u_+(t,
\cdot)$ and $u_-(t, \cdot)$ as in \eqref{decom}. Then for all
integers $m_1, m_2$ and $k_1, k_2$ with $m_1 + m_2 \leq m$ and $k_1
+ k_2 \leq k$, we have
\begin{equation}\nonumber
\int D^{m_1}\partial_t^{k_1}u_+\cdot D^{m_2}\partial_t^{k_2}u_-dx
\equiv 0.
\end{equation}
\end{prop}
\begin{proof}
Without loss of generality, we may assume that $m_2 < m$. By
Proposition \ref{prop-1}, one has
\begin{equation}\nonumber
u_+ = D^{-1}\nabla\times u_+.
\end{equation}
Consequently, one has
\begin{eqnarray}\nonumber
&&\int D^{m_1}\partial_t^{k_1}u_+\cdot
  D^{m_2}\partial_t^{k_2}u_-dx\\\nonumber
&&= \int D^{m_1}\partial_t^{k_1}D^{-1}\nabla\times u_+\cdot
  D^{m_2}\partial_t^{k_2}u_-dx\\\nonumber
&&= \int D^{m_1}\partial_t^{k_1}D^{-1} u_+\cdot
  D^{m_2}\partial_t^{k_2}\nabla\times u_-dx\\\nonumber
&&= - \int D^{m_1}\partial_t^{k_1}D^{-1} u_+\cdot
  D^{m_2 + 1}\partial_t^{k_2} u_-dx\\\nonumber
&&= - \int D^{m_1}\partial_t^{k_1} u_+\cdot
  D^{m_2}\partial_t^{k_2} u_-dx,
\end{eqnarray}
which implies the result in the proposition.
\end{proof}

We are ready now to prove our structural theorem, i.e. Theorem
\ref{thm-coercive}, for the helicity of solutions to the
incompressible Navier-Stokes equations.
\begin{proof}
We first notice that by integration by parts and the divergence-free property of $u$, there holds
\begin{eqnarray}\nonumber
\frac{d}{dt}\int u\cdot\omega dx &=& \int u_t\cdot\omega dx + \int (\nabla\times u_t) \cdot u dx\\\nonumber
&=& 2\int u_t\cdot\omega dx,\quad {\rm for}\ \omega = \nabla\times u.
\end{eqnarray}
Consequently,  there holds the following identity for helicity of the
incompressible Navier-Stokes equations \eqref{NS}:
\begin{eqnarray}\label{4-2}
\frac{d}{dt}\int u\cdot\omega dx = 2\nu\int \Delta u\cdot\omega dx.
\end{eqnarray}

Applying the decomposition in \eqref{decom} and using Proposition
\ref{prop-1}, we obtain that
\begin{eqnarray}\nonumber
\int u\cdot\omega dx &=& \frac{1}{4}\int (u_+ +
  u_-)\cdot(\nabla\times u_+ + \nabla\times u_-)dx\\\nonumber
&=& \frac{1}{4}\int (u_+ +
  u_-)\cdot(Du_+ - D u_-)dx.
\end{eqnarray}
By Proposition \ref{prop-2}, we further deduce that
\begin{eqnarray}\label{4-3}
\int u\cdot\omega dx &=& \frac{1}{4}\int (u_+\cdot Du_+ - u_-\cdot D
  u_-)dx\\\nonumber
&=& \frac{1}{4}\big(\|D^{\frac{1}{2}}u_+\|_{L^2}^2 -
  \|D^{\frac{1}{2}}u_-\|_{L^2}^2\big).
\end{eqnarray}
Similarly, we have
\begin{eqnarray}\label{4-4}
\int \Delta u\cdot\omega dx = -
\frac{1}{4}\big(\|D^{\frac{3}{2}}u_+\|_{L^2}^2 -
  \|D^{\frac{3}{2}}u_-\|_{L^2}^2\big).
\end{eqnarray}
Plugging \eqref{4-3} and \eqref{4-4} into \eqref{4-2}, we arrive at
\begin{eqnarray}\nonumber
&&\frac{d}{dt}\big(\frac{1}{2}\|D^{\frac{1}{2}}u_+(t)\|_{L^2}^2 +
\nu\int_0^t\|D^{\frac{3}{2}}u_+(s)\|_{L^2}^2ds\big)\\\nonumber &&=
\frac{d}{dt}\big(\frac{1}{2}\|D^{\frac{1}{2}}u_-(t)\|_{L^2}^2 + \nu
\int_0^t\|D^{\frac{3}{2}}u_-(s)\|_{L^2}^2ds\big).
\end{eqnarray}
Integrating the above differential inequality with respect to time,
one can complete the proof of the theorem.
\end{proof}

\section{Decay Properties of Data}

First of all, let us rewrite the data in \eqref{data-1} as
\begin{eqnarray}\label{3-1}
g(x) &=& \frac{1}{2}\int_{1 - \delta < |\xi| < 1 + \delta}\big(-
  in(\xi) + |\xi|^{- 1}\xi\times n(\xi)\big)
  \alpha(\xi)e^{ix\cdot\xi}d\xi\\\nonumber
&&+\ \frac{1}{2}\int_{1 - \delta < |\xi| <  1
  + \delta}\big(in(\xi) + |\xi|^{- 1}\xi\times n(\xi)\big)
  \alpha(\xi)e^{- ix\cdot\xi}d\xi\\\nonumber
&=& \frac{1}{2}\mathcal{F}^{-1}\big[\big(- in(\xi) + |\xi|^{-
  1}\xi\times n(\xi)\big)\alpha(\xi)\big](x)\\\nonumber
&& +\ \frac{1}{2}\mathcal{F}^{-1}\big[\big(in(\xi) + |\xi|^{-
  1}\xi\times n(\xi)\big)\alpha(\xi)\big](- x).
\end{eqnarray}
It is easy to check that
\begin{equation}\label{3-3}
\nabla\cdot g = 0,\quad \nabla \times g = Dg.
\end{equation}
Hence, one has
\begin{equation}\label{3-5}
g_- = 0,\quad g = g_+.
\end{equation}
Moreover, there holds
\begin{equation}\label{3-15}
\|g\|_{L^\infty} + \|\nabla g\|_{L^\infty} \lesssim
\|\widehat{g}\|_{L^1} + \|\widehat{\nabla g}\|_{L^1} \lesssim A.
\end{equation}

Next, we study the spatial decay properties of $g$ given in
\eqref{data-1}. For each $x$ with $|x| \neq 0$, let $B(x)$ be an
orthogonal matrix such that
\begin{equation}\nonumber
x = B\overline{x},\quad \overline{x} = \begin{pmatrix}0 \\
0 \\ |x|\end{pmatrix}.
\end{equation}
We use the sphere coordinate to parameterize $y$ as follows:
\begin{equation}\nonumber
B^T y = \widetilde{n}(y) = \begin{pmatrix}\sin\phi\cos\theta \\
\sin\phi\sin\theta \\ \cos\phi\end{pmatrix}, \quad 0 \leq \phi \leq
\pi, - \pi \leq \theta \leq \pi.
\end{equation}
We compute that
\begin{eqnarray}\nonumber
g(x) &=& \int\lambda^2d\lambda\int_{\mathbb{S}^2}\big[n(y)\sin<
  \widetilde{n}(y), \overline{x}>\\\nonumber
&&\quad +\ y \times n(y) \cos<
  \widetilde{n}(y), \overline{x}>\big] a(\lambda y) d\sigma_y\\\nonumber
&=& \int_0^\pi\int_0^{2\pi}\big[n(y) \sin(\kappa|x|\cos\phi)
  + y\times n(y)\cos(|x|\cos\phi)\big]\sin\phi d\theta a(\lambda y)
  d\phi\\\nonumber
&=& \frac{1}{|x|}\int_0^{2\pi} d\theta
  \int_0^\pi a(\lambda y)\big[n(y) d\cos(|x|\cos\phi) - y \times
  n(y) d\sin(|x|\cos\phi)\big]  \\\nonumber
&=& \frac{1}{|x|}\int_0^{2\pi}   a(\lambda y)
  \big[n(y) \cos(|x|\cos\phi) - y \times
  n(y) \sin(|x|\cos\phi)\big]\big|_{\phi = 0}^{\phi = \pi}d\theta\\\nonumber
&&\quad -\ \frac{1}{|x|}\int_0^{2\pi} d\theta
  \int_0^\pi\big[\cos(|x|\cos\phi) d\big(n(y)a(\lambda y)\big)\\\nonumber
&&\quad\quad\quad\quad -\ \sin(|x|\cos\phi) d\big( a(\lambda y)
y\times
  n(y)\big)\big].
\end{eqnarray}
Since, by \eqref{3-15}, $|\partial_\phi[a(\lambda y)y\times n(y)]| +
|\partial_\phi[n(y)a(\lambda y)]| \lesssim A$,
 one has
\begin{equation}\nonumber
|g(x)| \leq \frac{A}{|x|}.
\end{equation}
A similar bound can also be verified for $\nabla g$. Hence, we have
\begin{equation}\label{2.2}
|g(x)| + |\nabla g(x)| \lesssim \frac{A}{1 + |x|}.
\end{equation}

A similar calculation as above shows that $g_0$ given in
\eqref{data-2} satisfies
\begin{equation}\nonumber
\nabla\cdot g_0 = 0,\quad \nabla\times g_0 = g_0,\quad |g_0(x)| +
|\nabla g_0(x)| \lesssim \frac{1}{1 + |x|}.
\end{equation}

\section{Constructing Solutions by Cut-off and Perturbation}

In this section we construct the global smooth solutions to the 3D
Navier-Stokes equations with finite energy using the standard
cut-off and perturbation arguments.

First of all, let $v(t, x)$ be the solution of the heat equation
\begin{equation}\label{HE}
v_t = \nu\Delta v,\quad v(0, x) =v_0.
\end{equation}
If $v_0 = g$ which is given in \eqref{data-1}, then by \eqref{3-3}
(it is preserved by the heat flow in \eqref{HE}), one has
\begin{eqnarray}\label{NS-2}
&&v_t + v\cdot\nabla v + \nabla (- \frac{1}{2}|v|^2) -
  \nu\Delta v\\\nonumber
&&= - v\times (\nabla\times v) = - v\times Dv\\\nonumber &&= - v
\times (D - 1)v.
\end{eqnarray}
In this case, $v$ is a solution of the Navier-Stokes equations with
a forcing term $- v \times (D - 1)v$. Clearly, one has the following
estimate:
\begin{equation}\label{3-4}
|v(x)| + |\nabla v(x)| \lesssim \frac{Ae^{-\nu t/2}}{1 +
|x|}.
\end{equation}
Indeed, choosing $\beta \in C_0^\infty$ so that $\beta = 1$ on the support of $\alpha$ and $\beta = 0$ if $|\xi| \geq 1 + 2\delta$ and $|\xi| \leq 1 - 2\delta$, one has $v(t, x) = e^{-\nu t/2}\mathcal{F}^{-1}\big(e^{-\nu \sqrt{|\xi|^2 - 1/2}^2t}\beta(\xi)\big)\ast g(x)$. Note that $1/2 \leq |\xi|^2 - 1/2 \leq (1 + 2\delta)^2 - 1/2$ in the kernel $\mathcal{F}^{-1}\big(e^{-\nu \sqrt{|\xi|^2 - 1/2}^2t}\beta(\xi)\big)$. Then one can easily verify \eqref{3-4} by using \eqref{2.2}.

If $v_0 = g_0$ which is given in \eqref{data-2}, then the forcing
term in \eqref{NS-2} vanishes and the estimate \eqref{3-4} still
holds. So the proof of Theorem \ref{thm} can be carried out in the
same way as that of Theorem \ref{thm-1}. Below we will only present
the proof for Theorem \ref{thm-1}.

Suppose that $u$ is the unique local smooth solution of the
Navier-Stokes equations with initial data $u(0, x) = h_0 + \chi_M
g(x)$. Here $\|h_0\|_{H^1} \leq  M^{-\frac{1}{2}}$. The associated
pressure is $p = - \Delta^{-1}\nabla\cdot[\nabla\cdot(u\otimes u)]$.
To show that $u(t, x)$ is a global smooth solution, it is sufficient
to prove an \textit{a priori} estimate for $\|u(t, \cdot)\|_{H^1}$
for all $t > 0$. Define
\begin{equation}\nonumber
h = u - \chi_M v.
\end{equation}
It is easy to see that $h$ is governed by
\begin{equation}\label{h-eqn}
\begin{cases}
h_t + h\cdot\nabla h + \nabla p = \Delta
h + f,\\[-4mm]\\
\nabla\cdot h = - v\cdot\nabla\chi_M,\quad h(0, x) = h_0(x),
\end{cases}
\end{equation}
where $f$ is given by
\begin{eqnarray}\nonumber
&&f = - \chi_M(v_t - \Delta v) -
  h\cdot\nabla(\chi_M v) - \chi_M v\cdot\nabla h\\\nonumber
&&\quad +\  v\Delta\chi_M + 2(\nabla\chi_M\cdot\nabla) v  -
  \chi_M(v\cdot\nabla\chi_M)v - \chi_M^2v\cdot\nabla
  v\\\nonumber
&&=  v\Delta\chi_M + 2(\nabla\chi_M\cdot\nabla) v  -
  \chi_M(v\cdot\nabla\chi_M)v - \chi_M^2v\cdot\nabla
  v\\\nonumber
&&\quad -\ h\cdot\nabla(\chi_M v) - \chi_M v\cdot\nabla h.
\end{eqnarray}
Here and in what follows we will set $\nu$ to be 1.

Taking the $L^2$ inner product of \eqref{h-eqn} with $h$, we have
\begin{eqnarray}\nonumber
\frac{1}{2}\frac{d}{dt}\|h\|_{L^2}^2 + \|\nabla h\|_{L^2}^2 =
\frac{1}{2}\int|h|^2\nabla\cdot h dx + \int(p\nabla\cdot h + fh)dx.
\end{eqnarray}
Now let us use the expressions for $p$ and $f$ to rewrite that
\begin{eqnarray}\nonumber
&&\int(p\nabla\cdot h + fh)dx\\\nonumber &&=
  \int\Big(-\Delta^{-1}\nabla\cdot\big[h\cdot\nabla h - v\Delta\chi_M - 2(\nabla\chi_M\cdot\nabla) v
  + \chi_M v\cdot\nabla h + h\cdot\nabla(\chi_M v)\\\nonumber
&&\quad +\ \chi_M^2 v\cdot\nabla v + (\chi_M v) v\cdot\nabla
  \chi_M\big]\nabla\cdot h + \big[ v\Delta\chi_M + 2(\nabla\chi_M\cdot\nabla) v\\\nonumber
&&\quad -\ \chi_M(v\cdot\nabla\chi_M)v - \chi_M^2v\cdot\nabla
  v - h\cdot\nabla(\chi_M v) - \chi_M v\cdot\nabla
  h\big]h\Big)dx\\\nonumber
&&=  \int(h\cdot\nabla h)\Delta^{-1}\nabla\nabla\cdot hdx +
  \int(\chi_M^2v\cdot\nabla v)(\Delta^{-1}\nabla\nabla\cdot h -
  h)dx\\\nonumber
&&\quad +\ \int\Big(\big[\chi_M v\cdot\nabla h +
  h\cdot\nabla(\chi_M v) + (\chi_M v) v\cdot\nabla\chi_M\big]
  (\Delta^{-1}\nabla\nabla\cdot h - h)dx\\\nonumber
&&\quad +\ \int\big[-  v\Delta\chi_M +
  2\nabla_j(\nabla_j\chi_M v)\big](\Delta^{-1}\nabla\nabla\cdot h - h)\Big)dx.
\end{eqnarray}
Consequently, we have
\begin{eqnarray}\label{3-6}
&&\frac{1}{2}\frac{d}{dt}\|h\|_{L^2}^2 + \|\nabla
  h\|_{L^2}^2 = - \frac{1}{2}\int|h|^2v\cdot\nabla \chi_M dx\\\nonumber
&&\quad +\ \int(h\cdot\nabla h)\Delta^{-1}\nabla\nabla\cdot hdx +
  \int(\chi_M^2v\cdot\nabla v)(\Delta^{-1}\nabla\nabla\cdot h -
  h)dx\\\nonumber
&&\quad +\ \int\Big(\big[\chi_M v\cdot\nabla h +
  h\cdot\nabla(\chi_M v) + (\chi_M v) v\cdot\nabla\chi_M\big]
  \Delta^{-1}\nabla\times\nabla\times hdx\\\nonumber
&&\quad +\ \int\big[-  v\Delta\chi_M +
  2\nabla_j(\nabla_j\chi_M v)\big]\Delta^{-1}\nabla\times\nabla\times h\Big)dx.
\end{eqnarray}

We need estimate the right hand side of \eqref{3-6} term by term.
First of all, by \eqref{3-4}, by Sobolev imbedding inequality, it is
easy to see that
\begin{eqnarray}\nonumber
\Big|\frac{1}{2}\int|h|^2v\cdot\nabla \chi_M dx\Big| \lesssim
M^{-1}\|v\|_{L^\infty}\|h\|_{L^2}^2 \lesssim M^{-1}e^{-t/2}\|h\|_{L^2}^2.
\end{eqnarray}

Next, for the first term of the second line on the right hand side
of \eqref{3-6}, one can simply estimate that
\begin{eqnarray}\nonumber
\Big|\int(h\cdot\nabla h)\Delta^{-1}\nabla\nabla\cdot hdx\Big|
\lesssim \|h\|_{L^6}\|\nabla h\|_{L^2}\|\Delta^{-1}\nabla
  \nabla\cdot h\|_{L^3}
\lesssim \|h\|_{L^3}\|\nabla h\|_{L^2}^2.
\end{eqnarray}
Here we used the Sobolev imbedding $\|g\|_{L^6} \lesssim \|\nabla
g\|_{L^2}$ and the standard Calderon-Zygmund theory $\|Zg\|_{L^p}
\lesssim \|g\|_{L^p}$ for Riesz operator $Z$ and $1 < p < \infty$.
To treat the second term of the second line on the right hand side
of \eqref{3-6}, we first write that $$v\cdot\nabla v = - v\times
(\nabla\times v) + \frac{1}{2}\nabla |v|^2.$$  Using \eqref{NS-2},
we have
\begin{equation}\label{f}
\chi_M^2 v\cdot\nabla v = - \chi_M^2 v\times (D - 1)v -
\chi_M|v|^2\nabla\chi_M + \nabla(\frac{1}{2}\chi_M^2 |v|^2).
\end{equation}
Consequently, one has
\begin{eqnarray}\nonumber
&&\Big|\int(\chi_M^2v\cdot\nabla v)
  (\Delta^{-1}\nabla\nabla\cdot h - h)dx\Big|\\\nonumber
&&= \Big|\int\big[\chi_M^2 v\times (D - 1)v + \chi_M|v|^2
  \nabla\chi_M\big]\Delta^{-1}\nabla\times\nabla\times hdx\Big|\\\nonumber
&&\lesssim \big(M^{-1}\|v\|_{L^{12/5}(|x| \leq M)}^2 +
  \|(D - 1)v\|_{L^\infty}\|v\|_{L^{6/5}(|x| \leq M)}\big)\|
  \Delta^{-1}\nabla\times\nabla\times
  h\|_{L^6}\\\nonumber
&&\lesssim \big(M^{-1/2} +
  \|(|\xi| - 1)\widehat{v}\|_{L^1}M^{3/2}\big)e^{-t}\|
  \Delta^{-1}\nabla\times\nabla\times
  h\|_{L^6}\\\nonumber
&&\lesssim (M^{-1} + \delta^2M^3)e^{-2t} +
\frac{1}{16}\|\nabla
  h\|_{L^2}^2.
\end{eqnarray}

Now let us estimate the third line on the right hand side of
\eqref{3-6}. As above, we have
\begin{eqnarray}\nonumber
&&\Big|\int\big(\chi_M v\cdot\nabla h + h\cdot\nabla(\chi_M
  v)\big)\Delta^{-1}\nabla\times\nabla\times hdx\Big|\\\nonumber
&&\lesssim \big(\|\chi_M v\|_{L^\infty}\|\nabla h\|_{L^2}
  + \|\nabla(\chi_M v)\|_{L^\infty}\|h\|_{L^2}
  \big)\|\Delta^{-1}\nabla\times\nabla\times h\|_{L^2} \\\nonumber
&&\lesssim e^{-t/2}\|h\|_{L^2}^2 + \frac{1}{16}\|\nabla
  h\|_{L^2}^2,
\end{eqnarray}
and
\begin{eqnarray}\nonumber
&&\Big|\int(\chi_M v)v\cdot\nabla\chi_M\Delta^{-1}
  \nabla\times\nabla\times hdx\Big|\\\nonumber
&&\lesssim M^{-1}\|\sqrt{\chi_M}v\|_{L^{\frac{12}{5}}}^2
  \|\Delta^{-1}\nabla\times\nabla\times h\|_{L^6}\\\nonumber
&&\lesssim M^{-1}e^{- 2t} + \frac{1}{16}\|\nabla
  h\|_{L^2}^2.
\end{eqnarray}

For the last line,  we have
\begin{eqnarray}\nonumber
\Big|\int v\Delta\chi_M hdx\Big| \lesssim
M^{-2}\|v\|_{L^{\frac{6}{5}}(|x| \leq 2M)}\|h\|_{L^6} \lesssim
M^{-1} e^{- t} + \frac{1}{16}\|\nabla h\|_{L^2}^2.
\end{eqnarray}
Moreover, using integration by parts, we have
\begin{eqnarray}\nonumber
\Big|\int 2\nabla_j(\nabla_j\chi_M v)hdx\Big| \lesssim M^{-1}\|
v\|_{L^2(|x| \leq 2M)}\|\nabla h\|_{L^2} \lesssim M^{-1}e^{- t} + \frac{1}{16}\|\nabla h\|_{L^2}^2.
\end{eqnarray}

Inserting all the above estimates into \eqref{3-6}, we arrive at
\begin{eqnarray}\label{3-7}
&&\frac{1}{2}\frac{d}{dt}\|h\|_{L^2}^2 + (\frac{5}{16} -
  C\|h\|_{L^3})\|\nabla h\|_{L^2}^2\\\nonumber
&&\lesssim e^{-t/2}\|h\|_{L^2}^2 +
  (M^{-1} + \delta^2M^3)e^{- t}.
\end{eqnarray}

Now let us apply the ${\rm curl}$ operator to \eqref{h-eqn} and then
take the $L^2$ inner product of the resulting equation with ${\rm
curl}\ h$ to get
\begin{eqnarray}\label{3-8}
&&\frac{1}{2}\frac{d}{dt}\|\nabla\times h\|_{L^2}^2 + \|\nabla
  \nabla\times  h\|_{L^2}^2\\\nonumber
&&= - \int\nabla\times(h\cdot\nabla h)\nabla\times h dx + \int
  (\nabla\times f)\cdot(\nabla\times h)dx.
\end{eqnarray}
We first deal with the first term on the right hand side of
\eqref{3-8}. Using integration by parts and Hodge decomposition, we
estimate that
\begin{eqnarray}\nonumber
\Big|\int\nabla\times(h\cdot\nabla h)\nabla\times h
  dx\Big| &\lesssim& \|h\|_{L^3}\|\nabla h\|_{L^6}\|\nabla\times\nabla\times
  h\|_{L^2}\\\nonumber
&\lesssim& \|h\|_{L^3}\|\nabla\nabla\times h\|_{L^2}^2 +
  \|h\|_{L^3}\|\nabla\cdot h\|_{L^6}\|\nabla\times\nabla\times
  h\|_{L^2}.
\end{eqnarray}
Recall the second equation in \eqref{h-eqn}, one has
\begin{eqnarray}\nonumber
\|\nabla\cdot h\|_{L^6} = \|v\cdot\nabla\chi_M\|_{L^6} \lesssim
M^{-1}e^{- t/2}.
\end{eqnarray}
Using interpolation $\|h\|_{L^3} \lesssim
\|h\|_{L^2}^{\frac{1}{2}}\|\nabla h\|_{L^2}^{\frac{1}{2}}$, one
finally has
\begin{eqnarray}\nonumber
&&\Big|\int\nabla\times(h\cdot\nabla h)\nabla\times h
  dx\Big|\\\nonumber
&&\lesssim \|h\|_{L^3}\|\nabla\nabla\times
  h\|_{L^2}^2 + \frac{1}{16}\big(\|\nabla h\|_{L^2}^2 + \|\nabla\nabla\times
  h\|_{L^2}^2\big) + M^{-4}e^{- 2t}\|h\|_{L^2}^2.
\end{eqnarray}

For the second term on the right hand side of \eqref{3-8}, we first
write it as follows:
\begin{eqnarray}\label{3-9}
&&\int(\nabla\times f)\cdot(\nabla\times h)dx\\\nonumber &&= \int
  \nabla\times\big( v\Delta\chi_M +
  2\nabla\chi_M\cdot\nabla v
  - \chi_M(v\cdot\nabla\chi_M)v\big)\nabla\times h dx\\\nonumber
&&\quad +\ \int \nabla\times(\chi_M^2v\times (D - 1)v -
  \frac{1}{2}\chi_M^2\nabla|v|^2)\nabla\times h dx\\\nonumber &&\quad -\ \int
\nabla\times\big(h\cdot\nabla(\chi_M v) + \chi_M v
  \cdot\nabla h \big)\nabla\times h dx.
\end{eqnarray}
The first line on the right hand side of \eqref{3-9} is treated as
follows:
\begin{eqnarray}\nonumber
&&\Big|\int \nabla\times\big( v\Delta\chi_M +
  2\nabla\chi_M\nabla v - \chi_M(v\cdot\nabla\chi_M)v\big)\nabla\times h dx\Big|\\\nonumber
&&\lesssim \big(M^{-2}\|v\|_{L^2(|x| \leq 2M)} + M^{-1}\|\nabla
  v\|_{L^2(|x| \leq 2M)} + M^{-1}\|v\|_{L^4}^2\big)\|\nabla\nabla\times h\|_{L^2}\\\nonumber &&\leq
M^{-1}e^{- t} + \frac{1}{16}\|\nabla\nabla\times
h\|_{L^2}^2.
\end{eqnarray}
For the second term on the right hand side of \eqref{3-9}, we first
have
\begin{eqnarray}\nonumber
\Big|\int\nabla\times (\frac{1}{2}\chi_M^2\nabla|v|^2)\nabla\times
  hdx\Big|
&=& \Big|\int\nabla\times (\frac{1}{2}|v|^2\nabla
  \chi_M^2)\nabla\times hdx\Big|\\\nonumber
&\lesssim& M^{-2}e^{- 2t} +
  \frac{1}{16}\|\nabla\nabla\times h\|_{L^2}^2.
\end{eqnarray}
On the other hand, we estimate that
\begin{eqnarray}\nonumber
&&\Big|\int \nabla\times(\chi_M^2v\times (D - 1)v)\nabla\times h
  dx\Big|\\\nonumber
&&\lesssim \|v\|_{L^2(|x| \leq M)}\|(D - 1)
  v\|_{L^\infty}\|\nabla\nabla\times h\|_{L^2}\\\nonumber
&&\lesssim  \delta^2Me^{- 2t} +
\frac{1}{16}\|\nabla\nabla\times h\|_{L^2}^2
\end{eqnarray}
We estimate the last line on the right hand side of \eqref{3-9} as
follows:
\begin{eqnarray}\nonumber
&&\Big|\int \nabla\times\big( -  h\cdot\nabla(\chi_M v) - \chi_M v
  \cdot\nabla h\big)\nabla\times h dx\Big|\\\nonumber
&&\leq \big(\|h\|_{L^2}\|\nabla(\chi_M v)\|_{L^\infty} + \|\chi_M
  v\|_{L^\infty}\|\nabla h\|_{L^2}\big)\|\nabla\nabla\times h\|_{L^2}\\\nonumber
&&\lesssim e^{-t/2}\big(\|h\|_{L^2} + \|\nabla
  h\|_{L^2}\big)\|\nabla\nabla\times h\|_{L^2}\\\nonumber
&&\lesssim e^{- t}\big(\|h\|_{L^2}^2 + \|\nabla\times
  h\|_{L^2}^2 + \|v\cdot\nabla\chi_M\|_{L^2}^2\big) + \frac{1}{16}\|\nabla\nabla\times
  h\|_{L^2}^2\\\nonumber
&&\lesssim e^{- t}\big(\|h\|_{L^2}^2 + \|\nabla\times
  h\|_{L^2}^2\big) + M^{-1}e^{-2t}
  + \frac{1}{16}\|\nabla\nabla\times h\|_{L^2}^2.
\end{eqnarray}
We finally arrive at
\begin{eqnarray}\label{3-10}
&&\frac{1}{2}\frac{d}{dt}\|\nabla\times h\|_{L^2}^2 +
  \big(\frac{5}{16} - C\|h\|_{L^3}\big)\|\nabla \nabla\times  h\|_{L^2}^2\\\nonumber
&&\lesssim e^{- t}\big(\|h\|_{L^2}^2 + \|\nabla\times
  h\|_{L^2}^2\big) + \frac{1}{16}\|\nabla h\|_{L^2}^2\\\nonumber
&&\quad +\ (\delta^2M + M^{-1})e^{- t} .
\end{eqnarray}

Now let us add up \eqref{3-7} and \eqref{3-10} to yield that
\begin{eqnarray}\label{3-11}
&&\frac{1}{2}\frac{d}{dt}\big(\|h\|_{L^2}^2 +
 \|\nabla\times h\|_{L^2}^2\big) +
  \big(\frac{3}{8} - C\|h\|_{L^3}\big)\big(\|\nabla h\|_{L^2}^2
  + \|\nabla \nabla\times  h\|_{L^2}^2\big)\\\nonumber
&&\lesssim e^{- t/2}\big(\|h\|_{L^2}^2 + \|\nabla\times
  h\|_{L^2}^2\big) + (\delta^2M  + \delta^2M^3 + M^{-1})e^{- t}.
\end{eqnarray}
If there holds
\begin{eqnarray}\label{3-12}
\delta \lesssim M^{-2},\quad  C\|h\|_{L^3} \leq \frac{3}{8},
\end{eqnarray}
on some time interval $0 \leq t \leq T$, then \eqref{3-11} implies
\begin{eqnarray}\label{3-13}
\big(\|h(t, \cdot)\|_{L^2}^2 + \|\nabla\times h(t,
\cdot)\|_{L^2}^2\big) + \int_0^T\big(\|\nabla h\|_{L^2}^2 + \|\nabla
\nabla\times  h\|_{L^2}^2\big)ds \lesssim M^{-1}
\end{eqnarray}
on the same time interval. Since
\begin{eqnarray}\label{3-14}
C\|h\|_{L^3} \lesssim \big(\|h\|_{L^2} + \|\nabla\times h\|_{L^2} +
\|\nabla\cdot h\|_{L^2}\big) \lesssim M^{-\frac{1}{2}},\quad 0 \leq
t \leq T,
\end{eqnarray}
one sees that the second inequality in \eqref{3-12} is verified
provided that $M$ is sufficiently large. A standard continuation
argument simply implies that \eqref{3-13} holds for all time $t \geq
0$. Then one has
\begin{eqnarray}\label{3-14}
\|h\|_{H^1} \lesssim \big(\|h\|_{L^2} + \|\nabla\times h\|_{L^2} +
\|\nabla\cdot h\|_{L^2}\big) \lesssim M^{-\frac{1}{2}},
\end{eqnarray}
for all time $t \geq 0$. Since $u = h + \chi_M v$, one has $u \in
L^\infty(0, T; H^1)$, which is sufficient for the global
regularity of $u$.

%----------------------------------------------------------------------------acknowledgement
\section*{Acknowledgement}
Zhen Lei was in part supported by NSFC (grant No.11171072 and
11222107), NCET-12-0120, National Support Program for Young Top-Notch Talents, Shanghai
Shu Guang project
and Shanghai Talent
Development Fund. Fanghua Lin is partially supported by an NSF
grant, DMS-1065964 and DMS-1159313. Yi Zhou is partially supported
by the NSFC (No.11031001) and  "973 program" (grant No. 2013CB834100). Zhen Lei and Yi Zhou are also partially
supported by NSFC (No. 11421061) and SGST
09DZ2272900.

%-----------------------------------------------------------------------------bibliography

\end{document}